\documentclass[]{article}
\usepackage[a4paper,scale=.75,centering]{geometry}  

\usepackage{lipsum}
\usepackage{tikz}
\usepackage{amsfonts}
\usepackage{amsmath}
\usepackage{amsthm}
\usepackage{mathtools}
\usepackage{caption}
\usepackage{enumerate} 
\usepackage{xcolor}
\usepackage{subfig}
\usepackage{hyperref}

\newcommand{\E}{\mathbb{E}}        
        
\newcommand{\NN}{\mathbb{N}}

\newcommand{\p}{\mathbb{P}}         
   
\newcommand{\I}{\mathbb{1}}   
\newcommand{\B}{\mathcal{B}}

\newtheorem{theorem}{Theorem}[section]
\newtheorem{proposition}{Proposition}[section]
\newtheorem{remark}{Remark}[section]
\newtheorem{definition}{Definition}[section]
\newtheorem{corollary}{Corollary}[theorem]

\title{Lookdown construction for a Moran seed-bank model}

\author{
	M.~C. Fittipaldi \footnote{Facultad de Ciencias,  
		Universidad Nacional Aut\'onoma de M\'exico, M\'exico. \href{mailto:mcfittipaldi@ciencias.unam.mx}{mcfittipaldi@ciencias.unam.mx} }
	\and
	A. Gonz{\'a}lez Casanova \footnote{Instituto de Matemáticas, Universidad Nacional Aut\'onoma de M\'exico, M\'exico and Department of Statistics, University of California Berkeley. \href{mailto:gonzalez.casanova@berkeley.edu}{gonzalez.casanova@berkeley.edu},
	\href{mailto:adrian.gonzalez@im.unam.mx}{adrian.gonzalez@im.unam.mx}}
	\and 
	J.~E. Nava \footnote{Facultad de Ciencias,  
		Universidad Nacional Aut\'onoma de M\'exico, M\'exico. 
		\href{mailto:julionava3133@ciencias.unam.mx}{julionava3133@ciencias.unam.mx} }}
	
\begin{document}
		\maketitle
		\begin{abstract}
			We present a lookdown construction for a Moran seed-bank model with variable active and inactive population sizes and we show that the empirical measure of our model coincides with that of the Seed-Bank-Moran Model with latency of Greven, den Hollander and Oomen \cite{greven2022spatial}. 
			Furthermore, we prove that the time to the most recent common ancestor, starting from $N$ individuals with stationary  distribution over its state (active or inactive), has the same asymptotic order as the largest inactivity period. We then obtain an asymptotic distribution of the TMRCA, and use this result to find the first order of the asymptotic distribution of the fixation time of a single beneficial mutant conditioned to invade the whole population, which surprisingly is of order $\ln(N)$.
		\end{abstract}

		\section*{Introduction and main results}
		Seed-banks have been studied for their important effects in biology, but they are also of great interest from a mathematical point of view, since they can notably modify certain phenomena.  In the probabilistic community, efforts to study the effect of seed-banks rigorously can be traced back to the work of Kaj, Krone and Lascoux \cite{Kaj_Krone}. They postulated and studied an extension of the classical Wright-Fisher model which includes a seed-bank. The main difficulty of their approach is that the stochastic processes they study are not Markovian. Overcoming this (in a particular case) was achieved by Blath, Kurt, Wilke-Berenguer and one of the authors of this paper \cite{BGKW2016}. There, a new Wright-Fisher model with fully Markovian seed-bank was constructed. Moreover, the authors showed that the genealogy of this new model converges to the Seed-Bank coalescent. This process is a Kingman-type coalescent, where the lines of the coalescent can enter into latent state. As opposed to the Kingman coalescent, the seed-bank coalescent does not {\it comes down from infinity}. Also in \cite{BGKW2016} it was proved that the time to the most recent common ancestor of a sample of $N$ active individuals is of order $\ln(\ln(N))$. In \cite{lizz} the reader can find a detail description of the Seed-Bank coalescent.  
		
		A formidable tool to study coalescent processes is the {\it lookdown} construction by Donelly and Kurtz \cite{D&K_ClassicLD, DK1999}. This type of constructions allow to build Moran-type models where the lower $K$ levels form a Markov process, for every $K\in\mathbb N$.  This kind of constructions have played a prolific role in the study of many phenomena \cite{AE&K_GC, K&R_PRBP}, 
		Such representations are not known for seed-bank models and one main task of this paper is to fill this gap. 
		
		Our starting point is the Moran process with latency of Greven, den Hollander and Oomen, \cite{greven2022spatial}, in which active and dormant individuals are not two separate families of lines, but instead each line can change its state.
		
		In this work we present a lookdown construction for a Moran model with seed-bank with variable active population sizes and fixed total population size. This is a relaxation of assumption in \cite{BGKW2016} that the sizes of the active and dormant population are constant: we keep the total population size constant but allow the proportion of active (resp. dormant) individuals to vary over time, starting from a stationary distribution. We show that the empirical measure of our  SB-lookdown contruction and the empirical measure of the SB-Moran Model presented in \cite{greven2022spatial} coincide via the Markov Mapping theorem.

		Besides our novel lookdown process, we have two other main results. The first contribution is related to the time to the most recent common ancestor of the seed-bank coalescent. Thanks to \cite{BGKW2016}, it is known that time to the most recent common ancestor of the seed-bank coalescent starting from  $an$ active and $bn$ inactive individuals fulfills
		\begin{align*}
			\mathbb{E}\left[\widetilde{T}_{MRCA}[(an,bn)]\right] \asymp \ln(\ln(an) + bn).
		\end{align*}
		While this important result is concerned with expectations, we were able to find the distribution, up to its first order, of the time  to the most recent common ancestor $\widetilde{T}^{(n)}_{MRCA}$ starting from $n$ individuals with stationary state distribution. 
		
		To prove this result, we study the seed-bank coalescent as the ancestry process in the lookdown construction. This allowed us to compare $\widetilde{T}^{(n)}_{MRCA}$ with the largest period of inactivity of an individual, denoted by $\psi_n$, and to prove that they are of the same order.
		
		Finally, our last main result consists in finding the asymptotical distribution of the fixation time $\widetilde{\tau}^N$ of the type in the lower level for the $N$-lookdown model. The SB-lookdown construction allow us to compare the fixation time and the TMRCA, to show that 
		\begin{align*}
			\alpha\left(\widetilde{\tau}^N - \frac{\ln (Np)}{\alpha}\right)\xrightarrow[N\rightarrow \infty]{d} \widetilde{T},
		\end{align*}
		where $\widetilde{T} \sim Gumbel(0,1)$, and $p\coloneqq\frac{\alpha}{\alpha + \sigma}$ is the average time that a line is active.
		
		The remainder of the paper is structured as follows. In Section \ref{sec: SBlookdown} we construct the SB-lookdown and in Theorem \ref{thm:emp_eq} we show that the empirical measure of the SB-lookdown and SB-Moran Model coincide. In Section \ref{sec: TMRCA} we study the SB-coalescent as the ancestry process on the SB-lookdown, and in Theorem \ref{thm:line} we find the asymptotic distribution of the TMRCA. Finally, in Section \ref{sec : fixtime} we deduce the asymptotic distribution of the fixation time of the type in the first level in Corollary \ref{cor:fixation_time}.
		
		\section{Seed-Bank lookdown model}\label{sec: SBlookdown}
		Throughout this work we will use the framework of multitype Moran Seed-Bank model given by \cite[B.3]{greven2022spatial}, in which it was shown that the diffusion process as well as the ancestry processes associated with the Moran model coincide with those found in \cite{BGKW2016}. Let $N\in\NN$, the {\bf $N$-Moran Model with seed-bank}, or {\bf $N$-particle SB-Moran model}, describe a haploid population of $N$ individuals which evolve as follows. Each individual starts with a {\it type} in some set $E\subseteq \NN$ and a {\it state} in $S=\{a,d\}$ ({\it active} and {\it dormant} respectively), according with an exchangeable distribution. 
		Each active individual becomes inactive at rate $\alpha$ and each dormant individual becomes active at rate $\sigma$.
		Independently, a pair of individuals is uniformly chosen at rate $1$. The chosen individuals will interact if both are active.
		As a product of the interaction, one of them, chosen at random, will reproduce and its descendant will replace the other one. 
		We will denote the $N$-particle SB-Moran model at time $t\geq 0$ by $
		W^N(t)\coloneqq \left(W_1^N(t),\ldots,W_N^N(t)\right)$.
		Indeed, $W^N(t)$ is a random vector on $\left(E\times S\right)^{N}$.
		\begin{remark}
			Observe that the  dynamics mention above is completely symmetric. Therefore, given an initial exchangeable configuration, the SB-Moran model will be exchangeable.     
		\end{remark}
		We will now introduce an ordered particle model in the setting of \textit{classical} lookdown constructions \cite{D&K_ClassicLD} whose empirical measure distributions will agree with the SB-Moran model described above. Here, each particle will be attached to a level in $[N] = \{1,\ldots,N\}$. The population evolve as the SB-Moran Model except for the reproduction events, in which the parent will always be the individual with lower level involved in the reproduction event.		
		We will refer to this ordered particle model as the {\bf SB-lookdown model}, denoted at time $t\geq 0$ by    $ Z^N(t)\coloneqq \left(Z_1^N(t),\ldots,Z_N^N(t)\right)$,
		a random vector on $\left(E\times S\right)^{ N}$.

		In the following, we will formalize the lookdown construction and we will prove the equivalence between the ordered and unordered model via the Markov Mapping theorem \cite{AE&K_GC,Kurtz_mtg,K&R_PRBP}. 
		\subsection{Seed-Bank Lookdown construction}
		The first step in building the SB-lookdown model is to determine the initial activity and dormancy periods. 
		For each level $i \in \mathbb{N}$ we define
		\begin{equation*}
			\gamma^{i}(t):= y_0^ib_0^i + t +\sum\limits_{s_n \leq t}y_n =y_0^ib_0^i + t + \int_{0}^{t}\int_{0}^{\infty}y\mathcal{N}^{i}(ds,dy).
		\end{equation*}
		Here $y_0^i$ is an  exponential variable with parameter $\sigma$, $b_0^i$ is a Bernoulli variable with success probability $\sigma/\left(\alpha +\sigma\right)$, and   $\mathcal{N}^{i}(ds,dy)$ is a Poisson random measure with intensity measure $ \sigma ds \otimes \alpha e^{-\alpha y}dy,$ all of them independent of each other. Each pair $(s_n,y_n)$ indicates the start of a dormancy period at time $s_n$ with length $y_n$.

		Given $\tau^{i}(t) := \inf\{s:\gamma^{i}(s)>t\}$ the right-continuous inverse of $\gamma^{i}$, the activity and dormancy periods could be defined as
		\begin{equation*}
			a^{i} \coloneqq \{t \in \mathbb{R} \, : \quad  \tau^{i}(t) \neq \tau^{i}(s) \, \forall s \in \mathbb{R}\}
			\qquad \text{ and } \qquad 
			d^{i} \coloneqq  \mathbb{R}\setminus{a^{i}}.
		\end{equation*}

		Note that the state process for a given level only depends on itself. 
		Moreover, since the process starts form its stationary distribution we have that 
		\begin{align*}
			\p[t \in a^{i}] = \frac{\alpha}{\sigma+\alpha} =:p \qquad \mbox{for all} \quad  t>0, i \in \{1,\ldots,N\}.
		\end{align*} 
		The second part of this construction arises from considering the interactions among levels. 
		The times in which the individual on level $i$ tries to adopt the current type of an individual on the lower level $j$ are dictated by a Poisson process $\mathcal{C}^{i,j}$ with rate $1$, and the processes $\{\mathcal{C}^{i,j}\}_{j<i,i,j \in\mathbb{N}}$ are independent of each other. Hence, the type of the individual at level $i$ at time $t$ is determined by the last possible interaction with individuals at lower levels before that time. Namely, if  $(r_{n}^{ij})_{n \in \mathbb{N}}$ are the occurrence times of the process $\mathcal{C}^{i,j}$ and we set 
		\begin{align*}
			T_{t}^{ij}:= \sup_{n \in \mathbb{N}}\{r_{n}^{i,j}:r^{i,j}_{n} \in a^{i}\cap a^{j}, r_{n}^{i,j} \leq t \}
		\end{align*}
		the last interaction time between individuals on level $i$ and $j$ before time $t$.
		The type-state of the system at time $t$ is then provided by a function $g$ constructed in the following recursive way. 
		Given a random function $f:\mathbb{N} \rightarrow E$ which determine the initial type condition, we set  $g(1,t): = f(1)$.  For each level $i$,  given $g(1,t),\ldots,g(i-1,t)$,  we define
		\begin{align*}
			g(i,t): = f(i) \I_{\{T_{t}^i=0\}} + g(j,T^{i}_t) \I_{\{T_{t}^{i}  = T_{t}^{ij}\}},
		\end{align*}  
		where $T_{t}^{i}:= \sup \{T_{t}^{ij}, j<i\}$ is the last interaction time that change the type of the individual on level $i$. 
		Therefore, the state of the $N$-particle lookdown Moran seed-bank process at time $t$ is given by the vector  $(Z^N_{1}(t),\ldots,Z^N_{N}(t))$, where
		\begin{equation*}
			Z^N_i(t) \coloneqq \left(Z^{N,E}_i(t),Z^{N,S}_i(t)\right) = \left(g(i,t),a\I_{\{t \in a^{i}\}}+d\I_{\{t \in d^{i}\}}\right).
		\end{equation*}
		with $S = \{a,d\}$ the state space of each level. 
		The only differences in the dynamic of this process with the SB-Moran seed-bank model construction are found in the reproduction events. In the lookdown model, each individual choose his parent among the individuals in the lower levels, meanwhile in the  original SB model a pair of individuals is chosen to interact, and the parent is selected uniformly.

		\subsection{Equality in law of the SB-lookdown and SB-Moran Models} 
		In this (sub)section we will show that the empirical measure of the SB-lookdown and SB-Moran Model coincide. 
		\begin{theorem}
			\label{thm:emp_eq}
			The laws of empirical measures associated with the SB-Moran model and the SB-lookdown model coincide on  $D_{E\times S}[0,\infty)$.
		\end{theorem}
		The classical way to prove law equalities for lookdown constructions is via some coupling of processes (\cite[Section 2]{DK1999} ), but we will use the Markov Mapping theorem \cite[Theorem A.15]{K&R_PRBP}. To this end, let us introduce the generators of the empirical measures of the above mentioned processes. Let $\B\left((E\times S)^N\right)$ be the space  of  bounded measurable functions on $(E\times S)^N$. 
		Given a configuration $z\in \left(E\times S\right)^N$ and $f\in\B\left((E\times S)^N\right)$, we write 
		\begin{align}
			\label{gen:emp_M}
			A^{N}_Mf(z) \coloneqq   \sigma \sum_{i=1}^N  \left(
			f\left(\phi_{i}^d(z)\right)-f(z) 
			\right) + \alpha  \sum_{i=1}^N \left(
			f\left(\phi_{i}^a(z)\right)-f(z) 
			\right) + \frac{1}{2}\sum_{i\neq j }   \left(f(\phi_{ij}(z))-f(z) \right)
		\end{align}
		and  
		\begin{align}
			\label{gen:emp_LD}
			A^{N}_{LD}f(z) \coloneqq & \sigma \sum_{i=1}^N \left(
			f\left(\phi_{i}^d(z)\right)-f(z) 
			\right) + \alpha  \sum_{i=1}^N \left(
			f\left(\phi_{i}^a(z)\right)-f(z) 
			\right)  + \sum_{i< j } \left(f(\phi_{ij}(z))-f(z) \right)
		\end{align} 
		which correspond to the infinitesimal generator of the unordered and ordered models respectively. 
		Here, $\phi_{i,j}(z)$ corresponds to replace $z_j$ by $z_i$ only if $z_{i}^S = z_{j}^S=a$; $\phi_{i}^d(z)$ corresponds to replace $z_{i}^S$ by $d$; and $\phi_{i}^a(z)$ corresponds to replace  $z_{j}^S$ by $a$. 
		
		Let $\mathrm{z}_N$ be the empirical measure associated to $z$, given by $\mathrm{z}_N := \frac{1}{N}\sum\limits_{i=1}^{N} \delta_{z_i}$,
		which lives in the space $\mathcal{P}(E\times S)$ of probability measures in $E\times S$.
		For any  $f\in\B\left((E\times S)^N\right)$ we define 
		\begin{align}
			\label{op:avg}
			\widehat{f}(\mathrm{z}_N) := \frac{1}{N!} \sum_{\pi}f(z_\pi),
		\end{align}
		which correspond to  the uniform average out of all the permutations of  a given configuration such that its empirical measure is $\mathrm{z}_N$. 
		\subsubsection{Proof of Theorem \ref{thm:emp_eq}.}
			Since   
			\begin{align*}
				|A^N_{M}f(z)|\leq \left(2\sigma+2\alpha +N \right)N||f|| +1\quad \mbox{ and } \quad  |A^N_{LD}f(z)|\leq \left(2\sigma+2\alpha +N \right)N||f||+1
			\end{align*}
			for all $z\in \left(E\times S\right)^N$, with $||f|| = \sup_{z\in(E\times S)^N}|f(z)|$, to apply the Markov mapping theorem  it is enough to prove that the infinitesimal generators defined in equation \eqref{gen:emp_M} and equation \eqref{gen:emp_LD} coincide under the average operator given by \eqref{op:avg}.
			To this end, note that the terms correspondent to activation and dormancy events coincide. 
			First, for the deactivation mechanism we have that 
			\begin{align*}
				\frac{1}{N!}& \sum_{\pi}\sigma \sum_{i=1}^N  \left(
				f\left(\phi_{i}^d(z_\pi)\right)-f(z_\pi)  
				\right) =  \frac{1}{N!} \sum_{\pi}\sigma \sum_{i=1}^N  \left(
				f\left(\phi_{\pi^{-1}_i}^d(z)_\pi\right)-f(z_\pi)  
				\right) \\
				= & \frac{1}{N!} \sum_{\pi}\sigma \sum_{i=1}^N  \left(
				f\left(\phi_{i}^d(z)_\pi\right)-f(z_\pi)  
				\right)
				= \sigma \sum_{i=1}^N \left(
				\frac{1}{N!} \sum_{\pi}  f\left(\phi_{i}^d(z)_\pi\right)- \frac{1}{N!} \sum_{\pi}f(z_{\pi})  
				\right) \\
				=& \sigma  \sum_{i=1}^N \left(\widehat{f}\left(\mathrm{z}_N+ N^{-1}(\delta_{\phi_{i}^d(z)_i} -\delta_{z_i} )\right)-\widehat{f}(\mathrm{z}_N)\right),
			\end{align*}  
			and similarly the activation mechanism has the form
			\begin{align*}
				\frac{1}{N!} \sum_{\pi} \alpha  \sum_{i=1}^N \left(
				f\left(\phi_{i}^a(z_\pi)\right)-f(z_\pi)  
				\right) = & \alpha \sum_{i=1}^N \left(\widehat{f}\left(\mathrm{z}_N+ N^{-1}(\delta_{\phi_{i}^a(z)_i} -\delta_{z_i} )\right)-\widehat{f}(\mathrm{z}_N)\right).
			\end{align*}
			
			Finally, for the reproduction mechanism we have that 
			\begin{align*}
				\frac{1}{N!}&\sum_{\pi}\frac{1}{2} \sum_{i\neq j }  \left(f(\phi_{ij}(z_\pi))-f(z_\pi) \right) = \frac{1}{2} \sum_{i\neq j} \left(\widehat{f}(\mathrm{z}_N+N^{-1}\left(\delta_{z_i}-\delta_{z_j}\right))-\widehat{f}(\mathrm{z}_N)\right)\\
				=&  \frac{1}{2} \sum_{i< j}\left(\widehat{f}(\mathrm{z}_N+N^{-1}\left(\delta_{z_i}-\delta_{z_j}\right))-\widehat{f}(\mathrm{z}_N)\right)
				+  \frac{1}{2} \sum_{i>j}\left(\widehat{f}(\mathrm{z}_N+N^{-1}\left(\delta_{z_i}-\delta_{z_j}\right))-\widehat{f}(\mathrm{z}_N)\right)\\
				=& \sum_{i< j}\left(\widehat{f}(\mathrm{z}_N+N^{-1}\left(\delta_{z_i}-\delta_{z_j}\right))-\widehat{f}(\mathrm{z}_N)\right),
			\end{align*}
			and we conclude that 
			\begin{align*}
				A^N\widehat{f}(\mathrm{z}_N)  \coloneqq &\sigma  \sum_{i=1}^N \left(\widehat{f}\left(\mathrm{z}_N+ N^{-1}(\delta_{\phi_{i}^d(z)_i} -\delta_{z_i} )\right)-\widehat{f}(\mathrm{z}_N)\right)\\
				&+ \alpha \sum_{i=1}^N \left(\widehat{f}\left(\mathrm{z}_N
				+ N^{-1}(\delta_{\phi_{i}^a(z)_i} -\delta_{z_i} )\right)-\widehat{f}(\mathrm{z}_N)\right)\\
				& +\frac{1}{2} \sum_{i\neq j} \left(\widehat{f}(\mathrm{z}_N+N^{-1}\left(\delta_{z_i}-\delta_{z_j}\right))-\widehat{f}(\mathrm{z}_N)\right) \\
				= & \widehat{A^N_{LD}f(z)}=\widehat{A^N_{M}f(z)}.
			\end{align*}		
			 \hfill \qedsymbol	
		\section{Time to the most recent common ancestor} \label{sec: TMRCA}
		\subsection{  SB-coalescent process as the ancestry process of a SB-lookdown particle system}
		Let's recover the notation used for the Seed-bank coalescent on \cite{BGKW2016}.
		For $k\geq 1$ define $\mathcal{P}_{k}$ the set of partitions of $[k]$ and the set of marked partitions of $[k]$ as \begin{align*}
			\mathcal{P}_k^S=\{(\zeta,\overset{\rightarrow}{u}) | \zeta \in \mathcal{P}_k, \, \overset{\rightarrow}{u} \in S^{|\zeta|} \}.
		\end{align*} 
		If $\pi,\pi'\in\mathcal{P}_k^S$, we denote $\pi\succ \pi'$, if $\pi'$ can be obtained by merging exactly two blocks carrying the $a$-flag of $\pi$, and the resulting block also carries the $a$-flag. In a similar way, we denote $\pi\bowtie\pi'$ if $\pi$ can be constructed by changing the flag of precisely one block of $\pi$. 
		\begin{definition}[The Seed-bank $k$-coalescent]
			For $k\geq 2$ and $\alpha,\sigma\in (0,\infty)$. The seed-bank $k$-coalescent $\left(\Pi_t^{k}\right)_{t\geq 0}$ with seed-bank intensity $\alpha$ and relativity seed-bank size $\sigma$ is defined to be the continuous time Markov chain with values in $\mathcal{P}^{S}$, with the following transition rates:
			\begin{align*}
				\pi \mapsto \pi'  \text{ at rate } \begin{cases}
					1, & \text{if } \pi\succ \pi',\\
					\alpha, & \text{if } \pi \bowtie \pi' \text{ and one $a$-flag is replaced by one $d$-flag},\\
					\sigma, & \text{if } \pi \bowtie \pi' \text{ and one $d$-flag is replaced by one $a$-flag}.
				\end{cases}.
			\end{align*}			     
		\end{definition}
		Related to this coalescent we have, $\{(N_t,M_t)\}_{t\geq 0 }$ the block counting process associated to the seed-bank coalescent. Lets show that the ancestry process of a sample of individuals in the $N$-SB-lookdown model correspond with the SB-coalescent. 
		Suppose we sample $k$ individuals at a certain time $T\geq 0$. The initial condition correspond with the singletons of the sampled individuals levels marked according with its state at time $T$,  and we will recover its genealogical information tracing the particle system backwards in time. For each time $t\geq 0$, each block flag will correspond with the state of the lower individual level in the block at time $T-t$. Therefore, it is immediately noticeable that each block with $a$-flag turns into a $d$-flag block at rate $\sigma$ and conversely at rate $\alpha$. Similarly, coalescence events will correspond to the reproductions events between the lower individuals levels on each block.
		
		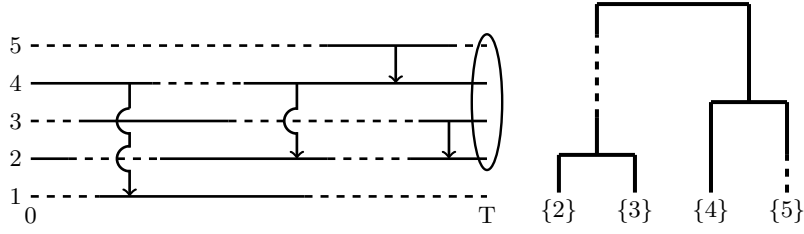
\begin{figure}[h]
			\centering
			\begin{subfloat}{
					\centering
					\begin{tikzpicture}[scale = 1]
						\draw[line width = 1pt,dashed] (-3,) -- (0.9,1);
						\draw[line width = 1pt] (0.9,1) -- (2.5,1);
						\draw[line width = 1pt,dashed] (2.5,1) -- (3,1);
						\draw[line width = 1pt] (-3,0.5) -- (-1.4,0.5);
						\draw[line width = 1pt,dashed] (-1.4,0.5) -- (-0.2,0.5);
						\draw[line width = 1pt] (-0.2,0.5) -- (3,0.5);
						\draw[line width = 1pt,dashed] (-3,0) -- (-2.3,0);
						\draw[line width = 1pt] (-2.3,0) -- (-0.4,0);
						\draw[line width = 1pt,dashed] (-0.4,0) -- (2.1,0);
						\draw[line width = 1pt] (2.1,0) -- (3,0);
						\draw[line width = 1pt] (-3,-0.5) -- (-2.5,-0.5);
						\draw[line width = 1pt,dashed] (-2.5,-0.5) -- (-1.3,-0.5);
						\draw[line width = 1pt] (-1.3,-0.5) -- (0.9,-0.5);
						\draw[line width = 1pt,dashed] (0.9,-0.5) -- (2,-0.5);
						\draw[line width = 1pt] (2,-0.5) -- (3,-0.5);
						\draw[line width = 1pt,dashed] (-3,-1) -- (-2.1,-1);
						\draw[line width = 1pt] (-2.1,-1) -- (0.6,-1);
						\draw[line width = 1pt,dashed] (0.6,-1) -- (3,-1);
						\draw[line width = 1pt] (0.5,0.5-2/6) arc (90:270:1/6);
						\draw[line width = 1pt] (0.5,0.5) -- (0.5,0.5-2/6-0.02);
						\draw[line width = 1pt,-to] (0.5,0.5-4/6+0.02) -- (0.5,-0.5);
						\draw[line width = 1pt] (-1.7,0.5-2/6) arc (90:270:1/6);
						\draw[line width = 1pt] (-1.7,-2/6) arc (90:270:1/6);
						\draw[line width = 1pt] (-1.7,-1/6+0.02) -- (-1.7,-2/6-0.02);
						\draw[line width = 1pt] (-1.7,0.5-2/6) -- (-1.7,0.5);
						\draw[line width = 1pt,-to] (-1.7,-4/6+0.02) -- (-1.7,-1);
						\draw[line width = 1pt,-to] (1.8,1) -- (1.8,0.5);
						\draw[line width = 1pt,-to] (2.5,0) -- (2.5,-0.5);
						\draw[line width= 0.9pt] (3,0.25) ellipse (0.2 and 0.9);
						
						\node at (-3.2,-1) (l1) {\small{1}};
						\node at (-3.2,-0.5) (l2) {\small{2}};
						\node at (-3.2,0) (l1) {\small{3}};
						\node at (-3.2,0.5) (l1) {\small{4}};
						\node at (-3.2,1) (l1) {\small{5}};
						\node at (-3,-1.25) (0) {\small{0}};
						\node at (3,-1.25) (T) {\small{T}};
					\end{tikzpicture}
				}
			\end{subfloat}
			\begin{subfloat}{
					\centering
					\begin{tikzpicture}[scale = 1]
						\draw[line width= 1.5pt] (5,-1.25) -- (5,-0.75);
						\draw[line width= 1.5pt] (6,-1.25) -- (6,-0.75);
						\draw[line width= 1.5pt] (7,-1.25) -- (7,-0.05);
						\draw[line width= 1.5pt,dashed] (8,-1.25) -- (8,-0.75);
						\draw[line width= 1.5pt] (8,-0.75) -- (8,-0.05);
						
						\draw[line width= 1.5pt] (5,-0.75) -- (6,-0.75);
						
						\draw[line width= 1.5pt] (7,-0.05) -- (8,-0.05);
						
						\draw[line width= 1.5pt] (5.5,-0.75) -- (5.5,0.-0.25);
						\draw[line width= 1.5pt,dashed] (5.5,-0.25) -- (5.5,0.85);
						\draw[line width= 1.5pt] (5.5,0.85) -- (5.5,1.25);
						\draw[line width= 1.5pt] (7.5,-0.05) -- (7.5,1.25);
						\draw[line width= 1.5pt] (5.5,1.25) -- (7.5,1.25);
						\node at (5,-1.5) (s2) {\small{\{2\} }};
						\node at (6,-1.5) (s2) {\small{\{3\}}};
						\node at (7,-1.5) (s2) {\small{\{4\}}};
						\node at (8,-1.5) (s2) {\small{\{5\}}};
					\end{tikzpicture}
				}
			\end{subfloat}
			\caption{\small{An illustration of the $5$-particle lookdown model and the ancestry process of a sample of individuals in levels $\{2,\ldots,5\}$ at time T. Straight lines indicates activity periods meanwhile dashed ones indicates the dormancy periods.}} 
			\label{fig:my_label}
		\end{figure}
		
		In this context,  we present the following result which will be helpful later and  it strongly relies in the lookdown construction presented before.   
		\begin{proposition}[First coalition time with k lower levels.]\label{prop:first_rep_event}
			Let $i,k \in \NN$, such that $k < i$. The first reproduction event between the individual at level $i$ with any of the individuals at first $k$ lower levels is given by
			\begin{align*}
				r^{i}_{k} \coloneqq \inf\left\{r_{n}^{ij}\text{ $|$ } r_{n}^{ij}\in A^{i} \cap A^{j},\,j\in\{1,\ldots,k\},\, n\in\NN\right\}
			\end{align*} 
			and satisfies 
			\begin{align*}
				r_k^i \sim \exp \left(k\left(\frac{\alpha}{\alpha+\sigma}\right)^2\right).
			\end{align*}
		\end{proposition}
		\begin{proof}
			Since the state process $Z^{N,S}_i$ starts from its stationary distribution recall that 
			\begin{align*}
				\mathbb{P}[r_{n}^{ij} \in A^{i}] = p =\frac{\alpha}{\alpha+ \sigma }, \text{ $\forall n \in \mathbb{N}$} 
			\end{align*}
			and by the independence between levels, 
			\begin{align}
				\label{proba_coal}
				\mathbb{P}[r_{n}^{ij} \in A^{i} \cap A^{j}] =  \mathbb{P}[r_{n}^{ij} \in A^{i}]\mathbb{P}[r_{n}^{ij} \in A^{j}] =  \left( \frac{\alpha}{\alpha+\sigma} \right)^{2} =p^2
			\end{align}
			for $j \in \{1,\ldots,k\}$.	Let $\widetilde{C}^{i,k}\coloneqq \sum\limits_{j=1}^{k} C^{i,j}$ be a Poisson process of rate $k$ by the superposition of Poisson Processes. Given $K$ be the number of occurrences of $\widetilde{C}^{i,k}$ until the first successful reproduction then we have that $K\sim Geo\left(p^2\right)$. From here, we conclude that 
			\begin{align*}
				\mathbb{P}[r^{i}_k\leq t] & = \sum_{n = 1}^{\infty} \mathbb{P}[K =n]\mathbb{P}\left[\widetilde{C}^{i,k}_t \geq n \right] = \sum_{n =1}^{\infty}  p^2(1-p^2)^{n-1} \int_{0}^{t}k \frac{1}{(n-1)!}(ks)^{n-1}e^{-ks}ds \\
				& = p^2k\int_{0}^{t}e^{-ks}\sum_{n=1}^{\infty}\frac{1}{(n-1)!}((1-p^2)ks)^{n-1}ds \\
				& = p^2k \int_{0}^{t}e^{-ks}e^{(1-p^2)ks}ds = 1-e^{-kp^2t}.
			\end{align*}  
		\end{proof}
		\begin{remark}~
			\label{exp_remarks}
			Note that the coalescence between the block that contain $i$ and the one that contain $j$ could be achieved without a reproduction event between their corresponding levels. This implies that if $\tau^i_k$ is the coalescent time of the block $\{i\}$ with any of the ancestral lines of the first $k$ individuals, then $\tau^i_k$ is bounded from above by $r_k^i$ almost surely.       
		\end{remark}
		
		\subsection{Asymptotic distribution of the time to the most recent common ancestor}
		We will study the long time behavior of the variable 
		\begin{align*}
			\widetilde{T}_{MRCA}[n] = \inf\left\{t\geq 0\,:\,N_t+M_t=1| N_0+M_0 = n\right\},
		\end{align*} using the framework developed in \cite{BGKW2016}. In particular, the authors define a stopping time $\varrho_n$ correspond to {\it ``the first time that all the $n$ initial individuals which so far had not entered the seed-bank have coalesced"}.
		We can bound the number of lineages left at $\varrho_n$ by  the number $B_n$ of first activation periods, that is
		\begin{align*}
			N_{\varrho_n}+M_{\varrho_n} \leq M_0 + \sum\limits_{i=2}^{N_0} \delta_i \leq  M_0 + \mathcal{B}_n,
		\end{align*}
		where  $\mathcal{B}_n:=\sum\limits_{i=2}^n \delta_i$, with $\delta_i \sim Be(2\sigma/(2\sigma +(i-1))$
		and are independent of each other, and also independent of $M_0$.  Note that
		
		\begin{equation} \label{eq: espvar}
			\begin{split}
				\E[M_0 + \mathcal{B}_n]&= pn +  \sum\limits_{i=2}^n\frac{2\sigma}{2\sigma+(i-1)}=pn + 2\sigma \ln n.+R(\sigma,n), 
				\qquad \mbox{and} \\
				\mathbb{V}(M_0 + \mathcal{B}_n) & = p(1-p)n + \sum\limits_{i=2}^n\frac{2\sigma}{2\sigma+(i-1)}\left(1-\frac{2\sigma}{2\sigma+(i-1)}\right)\leq E[M_0 + \mathcal{B}_n]
			\end{split}
		\end{equation}	
		with $R(\sigma,n)$ a function which converge to a finite value dependent of $\sigma$ as $n$ goes to infinity.
		Let $\psi_n$ be {\bf the largest inactivity period in the SB-coalescent process}, defined by
		\begin{equation}\label{eq: distpsin}
			\psi_n: = \max_{0\leq i\leq n}\{\delta_i \widetilde{\xi}_i\} \overset{d}{=} \max_{0\leq i\leq M_0 + \mathcal{B}_n} \{\xi_i\},
		\end{equation} 
		where the activation periods $\widetilde{\xi}_i,\ \xi_i\sim \exp(\alpha)$ are i.i.d.. We present some properties for $\psi_n$ below.
		\begin{proposition}	\label{prop_ldream_beh}
			\begin{enumerate}
				\item[]
				\item[a)] 
				Asymptotic distribution. 
				\begin{align} \label{eq: adpsi}
					\alpha\left(\psi_{n}-\frac{\ln np}{\alpha}\right)\xrightarrow[n\rightarrow \infty]{d} \widetilde{T},
				\end{align}
				where $\widetilde{T} \sim Gumbel(0,1)$.
				\item[]  
				\item[b)] Given $0<r<\frac{1}{\alpha}$, it holds that 
				$\quad \lim\limits_{n\rightarrow \infty} \p\left[\psi_{n} \leq r\ln np  \right] = 0.$ 
				\item[]  
				\item[c)]  The sequence $\{\psi_n\}_{n \in \NN}$ almost surely diverges.
			\end{enumerate}
		\end{proposition}
		\begin{proof}~
			\begin{enumerate}
				\item[a)]  
				Given $t>0$ and $\varepsilon>0$, using Tchebyshev's inequality we have that   
				\begin{equation}\label{eq: orderNn}
					\lim_{n\rightarrow \infty} \p\left[M_0 + \mathcal{B}_n< (p +\varepsilon)n \right]=1 \quad \text{ and } \quad  \lim_{n\rightarrow \infty} \p\left[M_0 + \mathcal{B}_n> (p-\varepsilon) n \right]=1
				\end{equation}
				We can write  
				\begin{align*}
					\lim_{n\rightarrow \infty} \p\left[\psi_{n}\leq \frac{t}{\alpha}+\frac{\ln np}{\alpha} \right] =  	\lim_{n\rightarrow \infty} \p\left[A_n \right],
				\end{align*}
				where
				\begin{align*}
					A_n = \left\{\psi_{n}\leq \frac{t}{\alpha}+\frac{\ln np}{\alpha}; \, (p-\varepsilon) n \leq M_0 + \mathcal{B}_n \leq (p+\varepsilon)n\right\}.
				\end{align*}
				Using the definition of $\psi_n$ we have that
				\begin{align*}
					\p\left[\max_{0\leq i \leq (p+\varepsilon) n} \xi_i \leq  \frac{t}{\alpha}+\frac{\ln np }{\alpha} \right]\leq
					\p\left[A_n \right]
					\leq \p\left[\max_{0\leq i \leq (p-\varepsilon)\ln n} \xi_i \leq  \frac{t}{\alpha}+\frac{\ln np}{\alpha}\right],
				\end{align*}
				and also  
				\begin{align*}
					\p\left[\max_{0\leq i \leq (p-\varepsilon) n} \xi_i \leq \frac{t}{\alpha}+\frac{\ln np}{\alpha} \right] & 
					= \left(1- \frac{e^{-t}}{np}\right)^{(p-\varepsilon)n}	\xrightarrow[n\rightarrow \infty] {}\exp\left[-\frac{(p-\varepsilon)e^{-t}}{p}\right].
				\end{align*}
				Analogously, 
				\begin{align*}
					\p\left[\max_{0\leq i \leq (p+\varepsilon) n} \xi_i \leq \frac{t}{\alpha}+\frac{\ln np}{\alpha} \right] 	\xrightarrow[n\rightarrow \infty]{} \exp\left[-\frac{(p+\varepsilon)e^{-t}}{p}\right],
				\end{align*}
				and letting $\varepsilon$ goes to zero we conclude that 
				\begin{align*}
					\lim_{n\rightarrow \infty} \p\left[\psi_{n}\leq \frac{t}{\alpha}+\frac{\ln np}{\alpha} \right] = \exp(-e^{-t}).
				\end{align*}
				The final observation is that $e^{-e^{-t}}$ corresponds to the cumulative distribution function of a standard Gumbel distribution.
				\item[b)]   
				Using equation \eqref{eq: distpsin} we have that 
				\begin{align*}
					\p\left[\psi_n \leq r\ln np\right] = & \ \p\left[\psi_n \leq r\ln np, \mathcal{B}_n +M_0 > (p-\varepsilon)n  \right]\\
					&  + \p\left[\psi_n \leq r\ln np, \mathcal{B}_n + M_0 \leq (p-\varepsilon)n \right]\\
					\leq & \  \p\left[\sup_{0\leq i \leq (p-\varepsilon)n} \xi_i \leq r\ln np \right]
					+ \p\left[\mathcal{B}_n + M_0 \leq (p-\varepsilon)n \right],
				\end{align*}
				and  thanks to equation \eqref{eq: orderNn} for the case $\varepsilon \in (0,p)$, and the fact that $\alpha r < 1$ by hypothesis, we conclude that
				\begin{align*}
					\p\left[\psi_n \leq r\ln np\right] =  \left(1- \tfrac{1}{(np)^{\alpha r}}\right)^{(p-\varepsilon)n}
					+ \p\left[\mathcal{B}_n + M_0 \leq (p-\varepsilon)n \right] \xrightarrow[n\rightarrow \infty]{} 0.
				\end{align*}
				\item[c)] 
				Given $M\geq 0$ and $r \in (0,\tfrac{1}{\alpha })$, there exist $K\in\NN$ such that  $M\leq r\ln np$ for all $n\geq K$. We have then that
				\begin{align*}
					\lim_{n\rightarrow \infty} \p[\psi_n \leq M] \leq \lim\limits_{n\rightarrow \infty} \p[\psi_{n} \leq r\ln np]=0,
				\end{align*}	
				so $\psi_n\xrightarrow{\p}\infty$. Since $\psi_n\overset{a.s.}{\leq} \psi_{n+1}$, we conclude that $\psi_{n}\xrightarrow{c.s.}\infty$.
			\end{enumerate}
		\end{proof}
		
		Using this properties, we now study the asymptotic behavior of the time to the most recent common ancestor.
		\begin{theorem}
			\label{thm:line}
			Given $(\Pi^n_t)_{t\geq 0}$ the seed-bank coalescent starting with $n$ individuals associated with the Moran lookdown seed-bank model, then  
			\begin{equation*}
				\frac{\widetilde{T}_{MRCA}[n]}{\psi_n}\xrightarrow[n\rightarrow \infty]{\p}  1. 
			\end{equation*}
			Moreover, we have that
			\begin{align*}
				\alpha\left(\widetilde{T}_{MRCA}[n] - \frac{\ln np}{\alpha}\right)\xrightarrow[n\rightarrow \infty]{d} \widetilde{T},
			\end{align*}
			where $\widetilde{T} \sim Gumbel(0,1)$.  
		\end{theorem}
		
		\subsubsection{Proof of Theorem \ref{thm:line}}
		Since $\psi_n$ is contained in the seed-bank coalescent, note that $\psi_n \overset{a.s.}{\leq} \widetilde{T}_{MRCA}[n]$. So, it is enough to prove that 
		\begin{align*}
			\frac{\widetilde{T}_{MRCA}[n]-\psi_{n}}{\psi_n}\xrightarrow[n\rightarrow \infty]{\p} 0.
		\end{align*} To this end, we will decompose the interval $[0,\widetilde{T}_{MRCA}[n]]$ in three different stages; the Kingman phase, the largest dream phase and the last coalescent phase. For the Kingman phase we considerate the time until $\varrho_n$, and we know that $\E[\varrho_n] \leq 2$ by $\cite{BGKW2016}$.
		
		On the other hand, let's analyze the number of lineages after the \textit{largest dormant phase}, the time interval between $\varrho_n$ and $\psi_n$.  Let $C_n$ be the number of lineages that did not coalesce with the first  $\ln (np)$ lineages before $\psi_n$. We will show that $C_n$ goes to zero almost surely when $n$ goes to infinity. 
		In fact, using Remark \ref{exp_remarks}, we have that 
		\begin{align*}
			C_n \coloneqq \sum_{i=\ln np}^{M_0 +\mathcal{B}_n} \I_{\left\{ \psi_n < \tau_{\ln np}^i\right\}}  \leq 
			\sum_{i=\ln np}^{M_0 + \mathcal{B}_n} \I_{\left\{ \psi_n < r_{\ln np}^i\right\}}, 
		\end{align*}
		and recalling definition \eqref{eq: distpsin} , taking conditional expectation with respect $\{M_0 + \mathcal{B}_n=m\}$ for any $m>\ln np$ we have that
		\begin{align*}
			\E&\left[\sum_{i=\ln np}^{M_0 + \mathcal{B}_n} \I_{\{\psi_{n}<r^i_{\ln np}\}}\Bigg|M_0 + \mathcal{B}_n=m\right] = (m-\ln np)\E\left[\I_{\{\psi_{n}<r^i_{\ln np}\}}\Bigg|M_0 + \mathcal{B}_n=m\right] \\
			& \leq  m\p\left[\max_{1\leq j\leq m} \xi_j< r_{\ln np}^i \right]
			= m\int_{0}^{\infty} \p\left[z<r_{\ln np}^i\right]m(1-e^{-\alpha z})^{m-1}\alpha e^{-\alpha z} dz\\
			& \leq m\int_{0}^{\infty} e^{-(\ln np) p^2z }(1-e^{-\alpha z})^{m-1}\alpha m e^{-\alpha z} dz
			= m\int_{0}^{\infty}  (np)^{-p^2z}(1-e^{-\alpha z})^{m-1}\alpha m e^{-\alpha z} dz \\
			& \leq m\E\left[(np)^{-p^2\psi_{n}}\Big|M_0 + \mathcal{B}_n=m\right].
		\end{align*}
		Now, using $L_2$-Cauchy inequality and the fact that $M_0 + \mathcal{B}_n \geq (p-\varepsilon)n > \ln np $ almost surely for $n$ sufficiently large by equation \eqref{eq: orderNn}, we have that 
		\begin{align*}
			\E[C_n] & \leq \E\left[\sum_{i=\ln np}^{M_0+ \mathcal{B}_n} \I_{\{\psi_{n}<r^i_{\ln np}\}}\right]
			\leq \E\left[(M_0 +\mathcal{B}_n)\E\left[(np)^{-p^2\psi_{n}}\Big|M_0 + \mathcal{B}_n\right]\right] \\
			& \leq  \E\left[\E\left[(M_0 + \mathcal{B}_n )(np)^{-p^2\psi_{n}}\Big|M_0 + \mathcal{B}_n\right]\right] 
			= \E\left[(M_0 +\mathcal{B}_n) (np)^{-p^2\psi_{n}}\right]\\
			&	\leq \left(\mathbb{V}(M_0 +\mathcal{B}_n)+\E^2[M_0 +\mathcal{B}_n]\right)^\frac{1}{2}\E\left[(np)^{-2p^2\psi_{n}}\right]^\frac{1}{2}\\
			& \leq  \left(\E[M_0 +\mathcal{B}_n]+\E^2[M_0 +\mathcal{B}_n]\right)^\frac{1}{2}\E\left[(np)^{-2p^2\psi_{n}}\right]^\frac{1}{2}\\
			& \leq  \left(np +2\sigma\ln n +R(\sigma,n)+\left(np + 2\sigma\ln n +R(\sigma,n)\right)^2\right)^\frac{1}{2}\E\left[(np)^{-2p^2\psi_{n}}\right]^\frac{1}{2},
		\end{align*}
		where in the last line we use equation \eqref{eq: espvar}.
		
		As $\{\psi_n\}_{n \in \NN}$ diverges almost surely, 
		$\lim\limits_{n\rightarrow \infty}\E[(np)^{k-2p^2\psi_{n}}]=0$ for any $k\geq 0$, so we can deduce that 
		$\lim\limits_{n\rightarrow 0} \E[C_n] = 0.$
		Moreover, as $C_n$ is non-negative, we conclude that
		$	\lim\limits_{n\rightarrow \infty} C_n \overset{a.s.}{=} 0.$
		As we can bound $\widetilde{T}_{MRCA}[n]$ by
		\begin{align*}
			\widetilde{T}_{MRCA}[n] \overset{c.s.}{\leq} \varrho_n+\psi_n+\widetilde{T}_{MRCA}\left[\ln np + C_n\right] \overset{c.s.}{\leq} \varrho_n+\psi_n+\widetilde{T}_{MRCA}\left[2\ln np\right]
		\end{align*} 
		for $n$ large enough,  we have that
		\begin{equation}\label{eq: boundlim}
			\frac{\widetilde{T}_{MRCA}[n]-\psi_n}{\psi_n}\overset{a.s.}{\leq} \frac{\varrho_n+\widetilde{T}_{MRCA}[2\ln np]}{\psi_n}.
		\end{equation}
		For any $\varepsilon>0$ and $r \in (0, \tfrac{1}{\alpha})$ we have 
		\begin{align*}
			\p\left[\frac{\widetilde{T}_{MRCA}[2\ln(np)]}{\psi_{n}}> \varepsilon \right]
			= & \p\left[\widetilde{T}_{MRCA}[2\ln(np)]> \varepsilon \psi_{n},\, \psi_{n} \leq r\ln(np) \right]\\
			&+\p\left[\widetilde{T}_{MRCA}[2\ln(np)]> \varepsilon \psi_{n},\, \psi_{n} > r\ln(np)) \right]\\
			\leq &\p\left[\psi_{n} \leq r\ln(np) \right]  + \p\left[\widetilde{T}_{MRCA}[2\ln(np)]> \varepsilon r\ln(np) \right]\\
			\leq & \p\left[\psi_{n} \leq r\ln(np) \right] +\frac{\E\left[\widetilde{T}_{MRCA}[2\ln(np)]\right]}{\varepsilon r\ln(np)}.
		\end{align*} 
		First term goes to zero according with Proposition \eqref{prop_ldream_beh}. We will show that the second term also goes to zero via the generalized dominated convergence theorem. By \cite[Remark 4.13]{BGKW2016} we know that
		\begin{align}
			\label{eq_limsup_tmrca}
			\limsup_{n\rightarrow \infty} \frac{\E\left[\widetilde{T}_{MRCA}[2\ln(np))]\right]}{\ln\left[2\ln (np)\right]}<\infty,
		\end{align}
		which implies that
		\begin{align*}
			\lim_{n\rightarrow\infty} \frac{\E\left[\widetilde{T}_{MRCA}[2\ln(np)]\right]}{\varepsilon r\ln(np)}=0.
		\end{align*}
		
		Finally, from \eqref{eq: boundlim} we conclude that
		\begin{align*}
			\frac{\widetilde{T}_{MRCA}[n]-\psi_n}{\psi_n}\xrightarrow[n\rightarrow \infty]{\p} 0.
		\end{align*}
		Moreover, using  the asymptotic distribution given by \eqref{eq: adpsi}, we can write for any $\varepsilon>0$
		\begin{equation*}
			\begin{split}
				\exp \left(-e^{-x}\right)=& \lim\limits_{n\rightarrow \infty} \p\left[\psi_n \leq \frac{x}{\alpha} + \frac{\ln(np)}{\alpha}\right]\\
				=& \lim\limits_{n\rightarrow \infty}\p\left[\psi_n \leq \frac{x}{\alpha} + \frac{\ln(np)}{\alpha}, \frac{\widetilde{T}_{MRCA}[n]}{\psi_n}>\varepsilon +1\right] \\&+ \lim\limits_{n\rightarrow \infty}\p\left[\psi_n \leq \frac{x}{\alpha} + \frac{\ln(np)}{\alpha}, \frac{\widetilde{T}_{MRCA}[n]}{\psi_n}\leq \varepsilon +1\right] \\
				\leq& \lim\limits_{n\rightarrow \infty}  \p\left[\widetilde{T}_{MRCA}[n]\leq (\varepsilon +1)\left(\frac{x}{\alpha} + \frac{\ln(np)}{\alpha}\right)\right],
			\end{split}
		\end{equation*} 
		and we can take $\varepsilon \rightarrow 0$ to obtain 
		\begin{align*}
			\lim\limits_{n\rightarrow \infty}  \p\left[\alpha \left(\widetilde{T}_{MRCA}[n] -\frac{\ln(np)}{\alpha}\right) \leq x\right]
			\geq \exp\left(-e^{-x}\right).
		\end{align*}
		In the other hand, since $\psi_n\overset{a.s.}{\leq} \hat{T}_{MRCA}(n)$ we have that
		\begin{align*}
			\lim\limits_{n \rightarrow \infty}  \p\left[\widetilde{T}_{MRCA}[n] \leq \frac{x}{\alpha} + \frac{\ln(np)}{\alpha}\right] \leq \lim\limits_{n \rightarrow \infty}  \p\left[\psi_n \leq \frac{x}{\alpha} + \frac{\ln(np)}{\alpha}\right] = \exp(-e^{-x}),
		\end{align*}
		and we conclude the desired convergence in distribution.

		\hfill \qedsymbol

		\section{Fixation time of a type in the SB-lookdown} \label{sec : fixtime}
		Suppose we start with an initial exchangeable distribution $\left(Z^N_1(0),\ldots,Z^N_N(0)\right)$ which sets at most one individual per type. Let $\widetilde{\tau}^N$ be the fixation time of type in the first level for the $N$-lookdown model, given by
		\begin{equation}
			\widetilde{\tau}^N \coloneqq \inf\left\{t\geq 0 \,:\, Z^{N,E}_i(t) = Z^{N,E}_1(0), \, i \in \{1,2,\cdots,N\}\right\}.
		\end{equation}
		If we study the the time to the most common recent ancestor $\widetilde{T}_{MCRA}[N]$ from a fixed initial time $\zeta \geq 0$, we see that the will only have the two following scenarios.
		\begin{enumerate}
			\item {\bf Scenario 1.} If $\widetilde{T}_{MCRA}[N]\leq  \zeta$, this implies that $(\zeta-\widetilde{T}_{MCRA}[N], \zeta)\subseteq (0,\zeta)$, and the fixation time $\widetilde{\tau}^N$ occurs also before $\zeta$, since all the individual already has coalesced with the individual in the first level. 
			\item {\bf Scenario 2.} If $\widetilde{T}_{MCRA}[N]>  \zeta$, this implies that $(0,\zeta) \subset (\zeta-\widetilde{T}_{MCRA}[N], \zeta)$ and the fixation time $\widetilde{\tau}^N$ is bigger than $\zeta$, because at time zero the individual in the first level has not yet coalesced  with all the $N-1$ remaining individuals.
		\end{enumerate}

		From before, we deduce that 
		$\p\left[ \widetilde{\tau}^N \leq \zeta \right] \leq \p\left[\widetilde{T}_{MCRA}[N] \leq \zeta \right]$ and  $\p\left[\widetilde{T}_{MCRA}[N] \leq \zeta \right] \leq \p\left[ \widetilde{\tau}^N \leq \zeta \right]$, 
		so 
		\begin{align*}
			\widetilde{\tau}^N \overset{d}{=}\widetilde{T}_{MCRA}[N] .
		\end{align*}
		Using this equality in law, we obtain the asymptotic distribution of the fixation time as a consequence of Theorem \ref{thm:line}. 
		\begin{corollary}\label{cor:fixation_time}
			Given  $\widetilde{\tau}^N$ the fixation time of the type in the first level for the $N$-lookdown model, it holds that
			\begin{align*}
				\alpha\left(\widetilde{\tau}^N - \frac{\ln Np}{\alpha}\right)\xrightarrow[N\rightarrow \infty]{d} \widetilde{T},
			\end{align*}
			where $\widetilde{T} \sim Gumbel(0,1)$.  
		\end{corollary}

		 \textbf{Acknowledgments.} This project was supported by the grant PAPIIT UNAM IN101722 “Nuevas aplicaciones de la dualidad de momentos y de la construcci\'on Lookdown”

	\end{document}